\documentclass[12pt, twoside, reqno] {amsart}

\usepackage{amsfonts}
\usepackage{amssymb}
\usepackage{latexsym}
\usepackage{graphicx,graphics}
\usepackage{epsf,graphicx,latexsym%
}

\newcommand{\be} {\begin{eqnarray}}
\newcommand{\ee} {\end{eqnarray}}
\newcommand{\bep} {\begin{eqnarray*}}
\newcommand{\eep} {\end{eqnarray*}}

\newcommand {\J}{\mathop{\mathcal{J}}\nolimits}

\newcommand {\Hol}{\mathop{\rm Hol}\nolimits}

\newcommand {\Id}{\mathop{\rm Id}\nolimits}
\renewcommand {\Im}{\mathop{\rm Im}\nolimits}
\renewcommand {\Re}{\mathop{\rm Re}\nolimits}

\newcommand {\JJ}{\mathcal{J}}

\newcommand {\A}{\mathcal{A}}
\newcommand {\BB}{\mathcal{B}}

\newcommand{\R}{{\mathbb R}}

\newcommand{\C}{{\mathbb C}}

\newcommand {\D}{\mathbb{D}}

\newtheorem{remar}{Remark}[section]
\newtheorem{examp}{Example}[section]

\newtheorem{corol}{Corollary}[section]
\newtheorem{propo}{Proposition}[section]
\newtheorem{theorem}{Theorem}[section]
\newtheorem{lemma}{Lemma}[section]

\newcommand{\rema}{\begin{remar}\rm}
\newcommand{\erema}{$\blacktriangleright$\end{remar}}

\newcommand{\exa}{\begin{examp}\rm}
\newcommand{\eexa}{$\blacktriangleright$\end{examp}}

\def\lwvec(#1 #2){\linewd 0.1
           \lvec(#1 #2)
           \linewd 0.05}

\makeatletter
\@namedef{subjclassname@2020}{\textup{2020} Mathematics Subject Classification}
\makeatother

\begin{document}

\title[Estimates of resolvents]{Estimates on some functionals over non-linear resolvents}

\author[M. Elin]{Mark Elin}

\address{Department of Mathematics,
         Ort Braude College,
         Karmiel 21982,
         Israel}

\email{mark$\_$elin@braude.ac.il}

\author[F. Jacobzon]{Fiana Jacobzon}

\address{Department of Mathematics,
         Ort Braude College,
         Karmiel 21982,
         Israel}

\email{fiana@braude.ac.il}


\begin{abstract}
Estimation of linear and quadratic functionals over different classes of univalent functions is one of the classical problems in geometric function theory. In this paper we solve the problem over some classes of so-called non-linear resolvents, which arise as a fruitful tool in dynamic systems. Sharp estimates on early Taylor coefficients and the Fekete--Szeg\"{o} functional are established.

\end{abstract}

\maketitle

\section{Introduction}\label{sect-intro}

Let $D$ be a domain in the complex plane $\C$.  Denote the set of holomorphic functions on $D$ by $\Hol(D,\C)$, and by $\Hol(D) := \Hol(D,D)$, the set of all holomorphic self-mappings of $D$.  We use the notion $D_r(c)$ for the open disk of radius $r$ centered at $c \in \C$ and denote $D_r:=D_r(0)$. Also we denote $\D=D_1$, the open unit disk.

Let $\Omega$ be the subclass of $\Hol(\D)$ consisting of functions vanishing at the origin:
\begin{equation}\label{def-U}
\Omega=\{ \omega \in \Hol(\D):\ \omega(0)=0 \}.
\end{equation}
The identity mapping on $\D$ will be denoted by $\Id$.

In this paper we deal with  the class of so-called non-linear resolvents. To define this class, we need the notion of semi-complete vector field.
Recall that a mapping $f\in\Hol(\D,\C)$ is called a semi-complete vector field on $\D$  if for every $z\in\D$ the Cauchy problem
\begin{equation}  \label{nS1}
\left\{
\begin{array}{l}
\frac{\partial u(t,z)}{\partial t}+f(u(t,z))=0,         \vspace{2mm} \\
u(0,z)=z%
\end{array}%
\right.
\end{equation}%
has a unique solution $u=u(t,z)\in\D$ for all $t\geq 0$. In this case, the unique solution of \eqref{nS1} forms a semigroup of holomorphic self-mappings of the open unit disk $\D$; see, for example, \cite{B-C-DM-book, E-R-Sbook, E-S-book, R-S1, SD}. The next theorem~gives criteria for a holomorphic function to be a semi-complete vector field.
 \begin{theorem}[see \cite{R-S1, SD, E-R-Sbook, E-S-book} for detail] \label{teorA}
 Let $f\in \Hol(\D , \C),\  f\not\equiv0$. The following statements are equivalent:
 \begin{enumerate}
 \item[(i)] $f$ is a semi-complete vector field on $\D$;

\item[(ii)] there exist a point $\tau\in \overline\D$ and a function $p\in\Hol(\D,\C)$ with ${\Re p(z)\ge0}$ such that
\begin{equation}\label{b-p}
f(z)=(z-\tau )(1-z\overline{\tau })p(z),\quad z\in\D,
\end{equation}
and this representation is unique;
\item[(iii)] $f$ satisfies the so-called range condition:
\[
\left(\Id +rf\right)(\D)\supset\D\qquad    \mbox{for all }\quad r>0,
\]
and $G_r:=(\Id+rf)^{-1}$ is a well-defined self-mapping of $\D$.
\end{enumerate}
 \end{theorem}
We notice that formula \eqref{b-p} is called the {\it Berkson--Porta representation} after the seminal work \cite{B-P} by Berkson and Porta. The mappings $G_r\in\Hol(\D),\ r>0,$ are called the {\it nonlinear resolvents} of a semi-complete vector field $f$, the net $\{G_r\}_{r>0}$ is the {\it resolvent family} for $f$. It is known that every non-linear resolvent is a univalent self-mapping of the open unit disk.

It follows from the uniqueness of the Berkson--Porta representation~\eqref{b-p} that every semi-complete vector field must have at most one null point in $\D $. Moreover, if the function $p$ in assertion (ii) satisfies $\Re p(z) > 0$, then this point is the \textit{Denjoy--Wolff point} for the semigroup  $\left\{u(t,\cdot)\right\} _{t\geq 0}$ defined by \eqref{nS1}  in the sense that
$ \tau= \lim\limits_{t\rightarrow \infty } u(t,z)$. Moreover, $\lim\limits_{r\to\infty}G_{r}(z)=\tau$ uniformly on compact subsets of $\D$.

In this paper we concentrate on the  case $\tau=0$.  Thus $f(z)=zp(z)$ with $\Re p(z)>0$  by the Berkson--Porta formula~\eqref{b-p}  and $\lim\limits_{r\to\infty} G_r(z)=0$.  In addition to their importance for dynamical systems (see \cite{E-R-Sbook, R-S1, SD}),  such resolvents form a very specific subclass of univalent self-mappings of the unit disk.

Recently, interesting geometric aspects of this class  were discovered in \cite{E-S-S} and \cite{E-J-21b}. For instance, it was established there that the resolvent family constitutes an inverse L{\oe}wner chain. Some covering and distortion results and the quasi-conformality of resolvents were proved. Further, the orders of starlikeness and of strong starlikeness were found. Also, it was shown there that the family of normalized resolvents converges to the identity mapping, uniformly on compact subsets of the unit disk.

The study of non-linear resolvents in the framework of geometric function theory naturally includes searching for sharp estimates on linear and non-linear functionals over this class as well as extremal functions. Note that the estimation of Taylor coefficients for different classes of analytic functions is a classical problem in  geometric function theory (see, for example, \cite{Dur}) starting from the famous Bieberbach conjecture.

Estimation of coefficient functionals over the class of non-linear resolvents is the aim of the current paper. More precisely, we consider the set of all semi-complete vector fields such that $f(0)=0$ with fixed derivative $f'(0)=q$. It can be seen from the Berkson--Porta representation~\eqref{b-p} that $\Re q \geq 0$. Since the case $\Re q=0$ is trivial, we assume that $\Re q>0$. We denote by $\JJ_r$ the set of non-linear resolvents $G_r=\left(\Id+rf\right)^{-1}$ of such vector fields. Our first problem is:

\vspace{1mm}

{\bf Problem 1:} {\it Find the sharp estimates and extremal functions for early Taylor coefficients over $\JJ_r$ and $\JJ:=\bigcup \JJ_r$.}

\vspace{1mm}

 We solve this problem for the second and third Taylor coefficients.

\vspace{1mm}

Concerning quadratic functionals, we  deal with the Fekete--Szeg\"{o}  functional that was introduced in  \cite{F-S},  found numerous applications and was studied by many mathematicians (see, for example, \cite[p. 124]{Pom} and \cite[p. 104]{Dur}). Recall that for an analytic function $h,$ $ h(z)=\sum\limits_{k=1}^{\infty}h_kz^k$ and $\lambda \in \C$, the Fekete--Szeg\"{o} functional is defined by
\begin{equation}\label{FS_funcl}
  \Phi(h, \lambda):=h_{1}h_{3}-\lambda h_{2}^2.
\end{equation}
It involves the Hankel determinant of second order  $H^2_1(f):=\Phi(h, 1)=\left|\begin{matrix}
                                      h_1 & h_2 \\
                                     h_2 & h_3
                                    \end{matrix}\right| $
and the Schwarzian derivative of $h$ at zero $\{h,0\}=\frac{6}{h_1^2}\cdot\Phi(h,1)$. The Fekete--Szeg\"{o}  problem is to find the sharp estimate on $\left| \Phi(\cdot, \lambda)\right|$ over a class of analytic functions.
Some general approaches to its solution were developed in \cite{ChKSug, Kanas}, see also reference therein. We estimate the Fekete--Szeg\"{o} functional over non-linear resolvents. Our aim is

\vspace{1mm}

{\bf Problem 2:} {\it Find the sharp estimates and extremal functions for the Fekete--Szeg\"{o}  functional over $\JJ_r$ and $ \JJ$.}

\vspace{1mm}
We solve this problem completely over $\JJ_r$ for every fixed $r>0$ and partially  over the union $ \JJ$.

Note that the sets $\JJ_r$ and $\JJ$ are invariant under the rotations $G\mapsto e^{-i \theta} G(e^{i \theta} \cdot)$. Therefore estimation of absolute values of the functionals in Problems 1 and 2 is equivalent to estimation of their real parts. 
It is worth mentioning that estimates on the second and third coefficients and  the Fekete--Szeg\"{o} functional over inverses for some families of functions were studied by many mathematicians. For some recent developments in this direction see, for example, \cite{YAI21}  and references therein.

Our approach is based on a conventional concept that by its definition  any resolvent family  consists of inverse functions for a one-parameter net $\{\Id+rf\}_{r>0}$. Hence we can use our previous results in \cite{E-J-21a} on families of inverse functions (see also \cite{MF20}).

Following the scheme suggested there, we introduce in the next section the subclass $\A_\psi$  of $\Hol(\D,\C)$ consisting of functions $F$ such that $F(z)/z$ takes values in a given half-plane and $F$ is invertible around zero. To be concrete, let fix a mapping $\psi \in \Hol(\D,\C),\  \psi(0)\neq0,$ of the open unit disk $\D$ onto a half-plane. Consider the class of  holomorphic functions
\begin{equation*}\label{classA}
\A_{\psi}:=\left\{F \in \Hol(\D,\C): \frac{F(z)}{z}\prec\psi \right\}.
\end{equation*}
In other words, $F \in \A_{\psi}$ if there exists a function $\omega \in \Omega$, see \eqref{def-U}, such that
\begin{equation}\label{subord}
F(z)=z\psi(\omega(z))\quad \text{ for all }\  z\in \D.
\end{equation}
We represent coefficients of such functions using determinants.

Clearly, every $F \in \A_\psi$ is locally univalent at the origin, and the inverse function $F^{-1}$ preserves $z=0$. Denote
\begin{equation*}\label{classB}
\BB_{\psi}:=\{F^{-1} : F\in \A_{\psi}\}.
\end{equation*}
We establish sharp estimates on early Taylor coefficients and the Fekete--Szeg\"{o} functional  over $\BB_\psi$.

It turns out that for an appropriate choice of the function $\psi$ the class $\BB_{\psi}$ coincides with $\JJ_r$.
This enables us in Section~\ref{sect-rigidity} to apply the previous results to non-linear resolvents and to solve Problems 1 and 2.
  Namely, Proposition~\ref{coro-resG} and Theorem~\ref{propo-estim-r} solve Problem 1 for $\JJ_r$, while Theorem~\ref{pr-5-4} solves it for $\JJ$. Theorem~\ref{th-phi-estim} solves Problem 2 for $\JJ_r$ and Theorem~\ref{th-appl for fekete} gives a partial solution of Problem 2 for the class $\JJ$. The results of this section complete and generalize earlier results obtained in~\cite{MF20}.

It is worth mentioning that in geometric function theory different coefficient functionals are studied as usual over families of  {\it normalized} univalent functions. By this reason we complete the paper presenting theorem that solves problems analogous to the problems above over the class of {\it normalized resolvents}.

\medskip

\section{Estimates over the classes $\A_\psi$ and $\BB_\psi$}
\setcounter{equation}{0}

We now focus on the case
\begin{equation}\label{psi2}
\psi(z)=\beta+\frac{\alpha z}{1-z}\,,
\end{equation}
where $\alpha, \beta \in\C$ with $\Re \frac{\beta}{\alpha} >0.$
Each such function $\psi$  maps the open unit disk $\D$ onto a half-plane.
The subordination relation $ \frac{F(z)}{z}\prec \psi(z)$ means
  \begin{equation*}\label{ineq}
 \A_{\psi}=\!\left\{ F : F(0)=F'(0)-\beta=0,\,  \Re \frac1\alpha\left(\frac{F(z)}{ z}-\beta\right) >-\frac 1 2\!\right\}\!.
  \end{equation*}

\begin{propo}\label{th_a_uxi}
   Let $F\in \A_{\psi}$, that is, $F(z)=z\left(\beta+\frac{\alpha\omega(z)}{1-\omega(z)}\right)$ with  $\omega(z)=\sum\limits_{n=1}^{\infty}c_nz^n\in\Omega.$ Then the Taylor coefficients of $F$ can be calculated by the formula
     \[
     a_p=\alpha\cdot \det\left(\begin{matrix}
                        c_1 & c_2 & c_3 & \ldots & c_{p-1} \\
                        -1 & c_1 & c_2 & \ldots & c_{p-2} \\
                        0&  -1& c_1 & \ldots & c_{p-3} \\
                       \vdots & {} & {} & {} & \vdots \\
                        0 & \ldots & 0 &  -1 & c_1
                      \end{matrix}\right),  \qquad p\ge2.
     \]
   \end{propo}
 \begin{proof}
 Write $ \psi $ in the form  $ \psi(z)=\beta-\alpha+\alpha\frac{1}{1-z}.$ Then each $F\in \A_{\psi}$ can be written as
\begin{equation*}\label{FbyFgal}
\frac{F(z)}{z}=\beta-\alpha+\alpha\frac{\widetilde{F}(z)}{z}\,, \quad  \text{ where }  \  \widetilde{F}(z)=\frac{z}{1-\omega(z)}\, .
\end{equation*}
Let $\widetilde{F}(z)=\sum\limits_{n=1}^{\infty}\widetilde{a}_nz^n$. Then $a_1=\beta-\alpha$ and $a_{p}=\alpha \widetilde{a}_p, \  p\geq2.$ Reminding that $\widetilde{a}_p$ can be calculated by Theorem 4.1 in \cite{E-J-21a}, we complete the proof.
\end{proof}

Let now turn to the class $\BB_\psi$.
\begin{examp}\label{ex-G}
Let $F \in \A_\psi$ with $\omega \in \Omega$ defined by $\omega(z)=e^{i \theta}z$, $\theta \in \R$. Then $G=F^{-1} \in \BB_\psi$ satisfies
\[
(\beta-\alpha) e^{i\theta} G^2(z)-(\beta+z e^{i\theta})G(z)+z \equiv0.
\]
Solve this equation taking in mind that  $G(0)=0$. Then we get
\begin{equation}\label{G-e-itheta}
G(z)=\frac{2z}{\beta+z e^{i\theta}+\sqrt{(\beta-z e^{i\theta})^2+4\alpha e^{i \theta}z}}\,.
\end{equation}
\end{examp}

\medspace

The following auxiliary assertion will be used to estimate the early coefficients and the Fekete--Szeg\"{o} functionals over the classes $\A_\psi$ and $\BB_\psi$.

\begin{lemma}\label{lem-omega}
Let $a,b \in \C$. Then for any  $\omega(z)=\sum\limits_{n=1}^{\infty}c_nz^n\in\Omega$  the following estimate holds:
\begin{equation}\label{estim12}
|a c_1^2+ b c_2| \leq \max\{|a|,|b|\}.
\end{equation}
This inequality is sharp. Specifically,
\begin{itemize}
  \item [ ]if $|a|>|b|$, equality in~\eqref{estim12} holds only for $\omega(z)=e^{i \theta }z$, $\theta \in \R$;
  \item [ ]if $|a|<|b|$, equality in~\eqref{estim12} holds only for $\omega(z)=e^{i \theta}z^2$, $\theta \in \R$;
  \item [ ]if $|a|=|b|$, equality in~\eqref{estim12} holds only for $\omega(z)=z\frac{\rho+e^{i \theta} z}{1+\overline{\rho}e^{i \theta} z}$, where $\rho \in \overline{\D}$ and $\theta \in \R.$
\end{itemize}
\end{lemma}

\begin{proof}
By the Schwarz Lemma $|c_1| \leq 1$. Also $|c_2| \le 1-|c_1|^2$ (see, for example, \cite{Dur}). Thus,
\begin{equation}\label{estim-1-2}
|a c_1^2+ b c_2| \leq |a|\cdot | c_1|^2+ |b|\cdot(1- | c_1|^2) \leq \max\{|a|,|b|\}.
\end{equation}
To find extremal functions, let separate the following cases.

If $|a|>|b|$, both signs in~\eqref{estim-1-2} are equalities only when $|c_1|=1$ and $c_2=0$. By the Schwarz Lemma $w(z)=e^{i \theta }z$, $\theta \in \R$, is the only extremal function in this case.

If $|a|<|b|$, both signs in~\eqref{estim-1-2} are equalities only when  $c_1=0$ and $|c_2|=1$. Then $\omega(z)=e^{i \theta}z^2$ with some $\theta \in \R$ is the only extremal function in this case.

Suppose $|a|=|b|$. The second sign in~\eqref{estim-1-2} is equality only when $\omega$ is the product of the identity mapping $\Id(z)=z$ with an automorphism~of~the~disk, that is, $c_2=(1-|c_1|^2)e^{-i\theta}$, $\theta \in \R$. The first sign in~\eqref{estim-1-2} becomes equality only if $\arg b-\theta=\arg a+2\arg c_1$. The proof is complete.
\end{proof}

Now we are ready to present sharp estimates for the early coefficients over the class $\BB_\psi$ as well as extremal functions. Assume that $F \in \A_\psi$ and
\begin{equation}\label{expanF-G}
G(z)=F^{-1}(z)=\sum_{k=1}^{\infty}b_kz^k.
\end{equation}

\begin{theorem}\label{th-from-ex}
 Let $F \in \A_\psi$ and $G=F^{-1}\in\BB_\psi$ have Taylor expansion \eqref{expanF-G}. Then $b_1=\frac{1}{\beta}$ and
\begin{itemize}
  \item [(1)]   $\left|b_2 \right| \leq \frac{|\alpha|}{|\beta|^3}$ with equality only for $F(z)=z\psi( e^{i \theta }z)$, $\theta \in \R$. In this case $G=F^{-1}$ is defined by \eqref{G-e-itheta}.

  \item [(2)] $\left|b_3\right| \leq \frac{|\alpha|}{|\beta|^5} \max \{|\beta|,|2\alpha-\beta| \}.$
Moreover,
\begin{itemize}
 \item [(i)] if $\Re \frac{\beta}{\alpha}<1$, then equality holds only for $F(z)=z\psi ( e^{i \theta }z)$  with some $\theta \in \R$;
  \item [(ii)] if $\Re \frac{\beta}{\alpha}>1$, then  equality holds only for $F(z)=z\psi(e^{i \theta}z^2)$ with some~$\theta \in \R$;
  \item [(iii)] if $\Re \frac{\beta}{\alpha}=1$, then equality holds only for $F(z)=z\psi\left(z\frac{\rho+e^{i \theta} z}{1+\overline{\rho}e^{i \theta} z}\right)$ with some $\rho \in \overline{\D}$ and $\theta \in \R.$
\end{itemize}
\end{itemize}
\end{theorem}

\begin{proof}
 Formula \eqref{psi2} implies $\psi(z)= \beta + \alpha\sum\limits_{n=1}^{\infty}z^n $. By \eqref{subord} we have $a_1=F'(0)=\beta$, hence $b_1=\displaystyle \frac{1}{\beta}$ by \eqref{expanF-G}.

Further, Proposition~\ref{th_a_uxi} gives $a_2=\alpha c_1$. Using the fact that $F\circ G=\Id$ we get  ${F''(0)G'(0)^2+F'(0)G''(0)=0}$. Then
\begin{equation}\label{b2}
b_2=-\frac{\alpha c_1}{\beta^3}.
\end{equation}
By the Schwarz Lemma $|c_1| \leq 1$ with equality only for rotations $ \omega(z)=e^{i \theta} z$. The explicit formula for $G$ was found in Example~\ref{ex-G}. Thus assertion (1) is proven.

Similarly, one sees that $a_3=\alpha (c_1^2+c_2)$ and $b_3=\frac{1}{\beta^5}(2a_2^2- a_3\beta)$, that is,
\begin{equation}\label{b3}
b_3=\frac{\alpha}{\beta^5}((2\alpha -\beta) c_1^2- \beta c_2).
\end{equation}
Denote $a:=2\alpha -\beta$ and $b:=-\beta$. Assertion (2) holds by Lemma~\ref{lem-omega}.
\end{proof}

\begin{remar}\label{rem-coeff-connec}
It follows from the proof of Theorem~\ref{th-from-ex} that ${|a_2|\leq |\alpha|}$ with equality only for $F(z)=z\psi(e^{i \theta }z)$, $\theta \in \R$, and ${|a_3|\leq 2|\alpha|}$ with equality only for  $\omega(z)=z\frac{\rho+e^{i \theta} z}{1+\overline{\rho}e^{i \theta} z}$ for some $\rho \in \overline{\D}$ and $\theta \in \R.$
Moreover, equalities $|a_2|=|\alpha|$ and $|b_2|=\frac{|\alpha| }{|\beta|^3}$ are equivalent and imply $|b_3|=\frac{|\alpha||2\alpha -\beta|}{|\beta|^5}\,$.
\end{remar}

To proceed we note that Proposition~\ref{th_a_uxi} enables to obtain estimates of the Fekete--Szeg\"{o} functional  $\Phi(\cdot, \lambda)$ defined by \eqref{FS_funcl} over the class~$\A_\psi$ by a straightforward calculation.
We provide estimates on $\Phi(\cdot, \lambda)$ over both classes $\A_\psi$ and~$\BB_\psi$, as well as describe extremal functions.

\begin{theorem}\label{tm-F-S-estim}
 Let $F\in\A_\psi$ and $G=F^{-1}\in\BB_\psi$. Then

$$\left|\Phi(G,\lambda)\right|\le \frac{|\alpha|}{|\beta|^6} \max (|\beta|, |\beta-(2-\lambda)\alpha|).$$ Moreover,
\begin{itemize}
  \item [(i)] if $|\beta-(2-\lambda) \alpha|>|\beta|$, then equality holds only for $F(z)=z\psi( e^{i \theta }z)$, $\theta \in \R$;
  \item [(ii)]  if $|\beta-(2-\lambda) \alpha|<|\beta|$, then equality holds only for $F(z)=z\psi(e^{i \theta}z^2)$, ${\theta \in \R}$;
  \item [(iii)] if $|\beta-(2-\lambda) \alpha|=|\beta|$, then equality holds only for $F(z)=z\psi\left(z\frac{\rho+e^{i \theta} z}{1+\overline{\rho}e^{i \theta} z}\right)$ with $\rho \in \overline{\D}$ and $\theta \in \R.$
\end{itemize}
\end{theorem}

\begin{proof}
Using formulas \eqref{b2} and \eqref{b3}, we calculate
$$\left|\Phi(G,\lambda)\right|=\frac{|\alpha|}{|\beta|^6}\left|(\beta-(2-\lambda)\alpha)c_1^2+\beta c_2\right|.$$
Denote $a:=\beta-(2-\lambda) \alpha $ and $b:=\beta$. Then the result follows from Lemma~\ref{lem-omega}.
\end{proof}

Note in passing that using early coefficients of $F$ that were found in the proof of Theorem~\ref{th-from-ex}, one sees
$$\left|\Phi(F,\lambda)\right|=|\alpha|\left|(\beta-\lambda\alpha)c_1^2+\beta\right|=|\beta|^6|\Phi(G,2-\lambda)|.$$
Therefore, in fact, Theorem~\ref{tm-F-S-estim}  provides also the sharp estimates on $\left|\Phi(\cdot,\lambda)\right|$ over $\A_\psi$, as well as the extremal functions.

To be more concrete, let $\Re\beta>0 $ and $\alpha=2\Re \beta$. Then the function $\psi$ maps the open unit disk onto the right half-plane. So, the last displayed formula can be transformed to estimates over the Noshiro--Warschawski class. In fact, this estimate
 coincides with the particular case $m=n=2$ of Theorem~3.2 in \cite{E-V}.

\begin{examp}\label{ex-starlike}
  Consider the set of all functions $F\in\Hol(\D,\C)$ such that $F(0)=F'(0)-1=0$ and $\Re\frac{F(z)}{z}\ge\frac12\,.$ This is equivalent to  $F\in \A_\psi$ with $\psi(z)=\frac1{1-z}$, that is, to the choice  $\alpha=\beta=1\,.$ It is well-known that this class contains the class $S^*(\frac12)$ of starlike functions of  order $\frac12$  which, in turn, contains the class $\mathcal{C}$ of  convex functions. Clearly, $\A_\psi$ is much wider since its elements are not necessarily univalent functions. Combination of Theorem~\ref{tm-F-S-estim} with results proven in \cite{Ke-Me} leads to
  \[
  \left| \Phi(F,\lambda)\right|\le \left\{ \begin{array}{ll}
                                               \max(\frac13,|1-\lambda|) & \mbox{ for }\ F\in\mathcal{C},\vspace{1mm} \\
                                               \max(\frac12,|1-\lambda|) &  \mbox{ for }\  F\in S^*(\frac12), \vspace{1mm}\\
                                               \max(1 ,|1-\lambda|) &  \mbox{ for }\ F\in \A_{\psi}.
                                             \end{array}                       \right.
  \]
\end{examp}

\medskip

\section{Estimates for nonlinear resolvents}\label{sect-rigidity}
\setcounter{equation}{0}

In this section we rely on previous results to establish estimates on the Taylor coefficients and the Fekete--Szeg\"{o}  functional over the class of nonlinear resolvents and to show that these estimates are sharp.

It turns out that for specific choices of $\psi$, the classes $\BB_\psi$ consist of non-linear resolvents. To make this clear, fix $q\in\C$ with $\Re q>0$. From now on we focus on the particular case $\alpha=2r\Re q$ and $\beta=1+rq$, that is,
\begin{equation*}\label{psi-r}
\psi_r(z)=1+r\,\frac{q+\overline{q}z}{1-z}=1+rq+2r\Re q\sum_{n=1}^\infty z^n.
\end{equation*}
Notice that in this case $\Phi_n(\psi_r,\lambda) =(2r\Re q)^2(1-\lambda)$ for $n\ge1$.

The following criteria  for a holomorphic function $F$ to belong to the class~$\A_{\psi_r}$  follows directly from our notations.
\begin{lemma}\label{lemmFinAr}
  Let $F\in\Hol(\D,\C),\ F(0)=0.$ The following conditions are equivalent:
  \begin{itemize}
  \item [(i)] $\Re \frac{F(z)}{z}\ge1$ for all $z\in\D$ and $F'(0)=1+rq$;
  \item [(ii)] $F(z)=z+rz\frac{q+\overline{q}\omega(z)}{1-\omega(z)}$ for some $\omega\in\Omega$;
  \item [(iii)] $F\in \A_{\psi_r}$.
\end{itemize}
\end{lemma}

Each  condition of this lemma is equivalent to the fact that the function $f$ defined by
\begin{equation} \label{G*-1}
 f(z)=\frac{F(z)-z}r=z\frac{q+\overline{q}\omega(z)}{1-\omega(z)}\,,\  \omega\in\Omega,
\end{equation}
is a semi-complete vector field on $\D$ by the Berkson--Porta formula. Hence the (right) inverse function $F^{-1}=:G(=G_r)$, which, in fact, solves the functional equation
\begin{equation*} \label{G*-2}
G_r+ rf\circ G_r=\Id
\end{equation*}
is holomorphic in the open unit disk $\D$ by Theorem~\ref{teorA}. (Recall that $G_r$ is a univalent self-mapping of $\D$, which is called   the resolvent of $f$, see Section~\ref{sect-intro}.)

Let construct several semi-complete vector fields by formula~\eqref{G*-1}. These examples will be useful  in the sequel.
\begin{examp}\label{ex-gen-calc}
Let $\omega\in\Omega$ and $f(z)=z\frac{q+\overline{q}\omega(z)}{1-\omega(z)}$.
\begin{itemize}
 \item [(i)] Choosing $\omega(z)=e^{i \theta }z$, $\theta \in \R$, we get $f_1(z)=\displaystyle z\frac{q+\overline{q}ze^{i \theta}}{1-ze^{i \theta}}\,.$
In this case
   \begin{equation}\label{Gr-rReg-gt1}
G_r(z)=\frac{2z}{ze^{i\theta}+rq+1+ \sqrt{(ze^{i\theta}-1-rq)^2+8r ze^{i\theta}\Re q}}\,.
   \end{equation}
Indeed, the form of $G_r$ in \eqref{Gr-rReg-gt1} can be obtain from Example~\ref{ex-G}.
  \item [(ii)] Choosing $\omega(z)=e^{i \theta}z^2$, $\theta \in \R,$ we get $ f_2(z)=\displaystyle z\frac{q+\overline{q} e^{i \theta}z^2 }{1-e^{i \theta}z^2 }\,$.
 \item [(iii)] Choosing $\omega(z)=z\frac{\rho+e^{i \theta} z}{1+\overline{\rho}e^{i \theta} z}$ with $\rho \in \overline{\D}$ and $\theta \in \R$, we get
\[
f_3(z)=\displaystyle z \frac{q+\left(\overline{q}\rho+\overline{\rho}qe^{i \theta}\right)z
+\overline{q}e^{i \theta}z^2}{1+(\overline{\rho} e^{i\theta} -\rho)z- e^{i \theta}z^2}\,.
\]
If $|\rho|=1$, then $f_3$ coincides with $f_1$,   while if $\rho=0$ then $f_3\equiv f_2.$
\end{itemize}
\vspace{1mm}
Throughout this section, we will write the functions $f_1, f_2, f_3$ without explicitly indicating their dependence on the parameters $\theta$ and $\rho$, although this dependence is implied.
Note that $f_1$, $f_2$ and $f_3$ can be obtained from extremal functions described in  Theorem~\ref{tm-F-S-estim} by the formula $f(z)=\frac{F(z)-z}r$~(see~\eqref{G*-1}).
 \end{examp}
We denote the by $\JJ_r$ set of nonlinear resolvents $G_r$ for all semi-complete vector fields $f$ normalized by $f(0)=0$ and $f'(0)=q$ (thus $\JJ_r=\BB_{\psi_r}$).

Let a nonlinear resolvent $G_r\in\JJ_r$ have the Taylor expansion $$G_r(z)=\displaystyle\frac z{1+rq}+\sum_{n=2}^\infty b_nz^n.$$

 Theorems~\ref{th-from-ex}--\ref{tm-F-S-estim} imply estimates on the Taylor coefficients and the Fekete--Szeg\"{o} functional over $\JJ_r.$ Indeed, let substitute $\beta=1+rq$ and $\alpha=2r\Re q$ into the inequalities in that theorems. Then using the notations
 \begin{eqnarray}\label{u-v}
\left.
 \begin{array}{lll}
 u_q(r)&:=&\displaystyle\frac{2r\Re q}{|1+rq|^5} \qquad \text{ and } \\
v_q(r)&:=&u_q(r)\cdot w_q(r)\quad\text{ with }\vspace{1mm}\\
 w_q(r)&:=& \displaystyle \frac{\left|1+rq-2(2-\lambda) r\Re q\right|}{|1+rq|}
\end{array}
\right\}
 \end{eqnarray}
one gets:
\begin{propo}\label{coro-resG}
     For every $G \in\JJ_r$ the following estimates hold:
\begin{eqnarray}
&& \left|b_2 \right| \leq u_q(r)\cdot |1+rq|^2, \label{b2est}\quad\\
&&\left|b_3 \right|\leq  u_q(r)\cdot\max \left( |1+rq|,\left|1+rq-4r \Re q\right|\right)\!,\label{b3est} \\
   &&  \left|\Phi(G,\lambda) \right| \le  \max\left(u_q(r) , v_q(r) \right).  \label{Phi-1}
\end{eqnarray}
  \end{propo}

Comparison of the estimates on $b_2$ and $b_3$ in this proposition with Theorem~\ref{th-from-ex} allows to get the description of functions maximizing $|b_2|$ and $|b_3|$ over $\JJ_r$. To this end we use semi-complete vector fields $f_1$, $f_2$ and $f_3$ introduced in Example~\ref{ex-gen-calc}.

\begin{theorem}\label{propo-estim-r}
Estimates \eqref{b2est} and \eqref{b3est} are sharp. Moreover, let $G\in\JJ_r$ be the resolvent of $f$, then
 \begin{itemize}
   \item [(i)] if $f=f_1$, then
             equality in  \eqref{b2est} holds for all $r>0$ and equality in  \eqref{b3est} holds  whenever $\, r\!\Re q>1$;
 \item [(ii)]  if either $f=f_2$ and $r\Re q< 1$, or $f=f_3$ and $r\Re q=1$, then inequality \eqref{b2est} is strong, while \eqref{b3est} becomes equality;
  \item [(iii)] otherwise, both inequalities \eqref{b2est} and \eqref{b3est} are strong for all ${r>0}$. \end{itemize}
\end{theorem}
Note that in the case (i) $G$ is defined by \eqref{Gr-rReg-gt1} and for $\, r\!\Re q>1$ we have $\left|b_3 \right|=  \frac{2r\Re q}{|1+rq|^5}\! \cdot \! \left|1+rq-4r \Re q\right|$.
\begin{proof}
Recall that $G \in \JJ_r $ is the resolvent of $f$ only if $G$ is the inverse function of $F=\Id+rf$ (cf. Lemma~\ref{lemmFinAr} and formula~\eqref{G*-1}).

Suppose that $r\Re q<1$.  Then by Theorem~\ref{th-from-ex}, equality in \eqref{b3est} holds only if $F(z)=z\psi_r(z^2e^{i \theta}),$
  $\theta \in \R$. In this case $f=f_2$  (cf. Example~\ref{ex-gen-calc}(ii)) and inequality in \eqref{b2est} is strong.

If $r\Re q=1$,  then by assertion (iii) of Theorem~\ref{th-from-ex}, equality in \eqref{b3est} holds only if $F(z)=z\psi\left(z\frac{\rho+e^{i \theta} z}{1+\overline{\rho}e^{i \theta} z}\right)$
with some $\rho \in \D$ and $\theta \in \R.$
In this case $f=f_3$  (cf. Example~\ref{ex-gen-calc}(iii)) and inequality in \eqref{b2est} is strong. Assertion (ii) is proven.

By assertions (1) and (2(i)) of Theorem~\ref{th-from-ex}, if $F(z)=z\psi_r(e^{i \theta }z)$ with some $\theta \in \R$,
 then equality  in \eqref{b3est} holds when $r\Re q>1$, while equality in  \eqref{b2est} holds for all $r>0$.
In this case $G$ fits to $f_1$  and is defined by \eqref{Gr-rReg-gt1} (cf. Example~\ref{ex-gen-calc}(iii)). So, assertion (i) is proven.
Moreover, we have $\left|b_3 \right|=  \frac{2r\Re q}{|1+rq|^5}\! \cdot \! \left|1+rq-4r \Re q\right|$ by  Remark~\ref{rem-coeff-connec}.

Otherwise, if $G \in \JJ_r$ is not resolvent of either $f_1$, or $f_2$, or $f_3$, then inequalities in \eqref{b2est} and \eqref{b3est} are strong for all $r>0$ due to Theorem~\ref{th-from-ex}.
This completes the proof.
\end{proof}

\begin{corol} Let $\{G_r\}_{r>0}$ be the resolvent family for a semi-complete vector field.
\begin{itemize}
  \item [(i)] Equality in \eqref{b2est} holds  for some $r>0$ if and only if it holds for all $r>0$.

  \item [(ii)] Equality in \eqref{b3est} holds for some $r<\frac1{\Re q}$ (respectively, $r>\frac1{\Re q}$) if and only if it holds for all $r<\frac1{\Re q}$  (respectively, $r>\frac1{\Re q}$).
\end{itemize}
\end{corol}

Recall that $\JJ_r$ is a subclass of the class of hyperbolically convex functions normalized by $G(0)=0$ and $G'(0)=\frac1{1+rq}$\,.  It follows from a result in \cite{Ma-Mi-94} that $|b_2|\le u_q(r) |1+rq|^2 +\frac{|rq|^2}{|1+rq|^3}$  for hyperbolically convex functions and this bound is sharp, while for the best of our knowledge the sharp bound on $b_3$ is unknown.

Search for extremal functions to the Fekete--Szeg\"o functional $\Phi(\cdot,\lambda)$ over $\JJ_r$ is a more complicated problem. To solve it, in addition to notations~\eqref{u-v}, we denote
\begin{equation}\label{c-ro-mu}
c:=\frac{3}{2}-i\frac{1}{2}\tan (\arg q), \quad \rho:=\frac{|q|}{2\Re q} \quad \text{ and } \quad \mu:=\frac{2-\Re \lambda }{\left|\lambda-c\right|^2-\rho^2}\,.
\end{equation}

\nopagebreak Define a partition of the complex plane $\C=\bigcup\limits_{i=1}^5 S_i$ as follows (see Fig.\ref{fig1}) 
\begin{equation}\label{partition}
\left.
\begin{array}{ccl}
 S_1& = &\left\{\lambda \in \C : \Re \lambda \geq2, \lambda \notin {D_{\rho}(c)}\cup\{2, 2-i\tan(\arg q) \} \right\}, \\
 S_2&=&\{\lambda \in \C : \Re \lambda \leq2, \lambda \in \overline{D_{\rho}(c)}\setminus\{ 2, 2-i\tan(\arg q)\} \}, \\
S_3 &=&\{2,2-i\tan(\arg q)\}, \\
S_4&=&\{\lambda \in \C : \Re \lambda >2 \text{ and } \lambda \in {D_{\rho}(c)} \}, \\
S_5&=& \{\lambda \in \C : \Re \lambda <2 \text{ and } \lambda \notin \overline{D_{\rho}(c)} \}.
\end{array}
\right\}
\end{equation}

\begin{figure} [t]
  \centering
  \includegraphics[width=5cm]{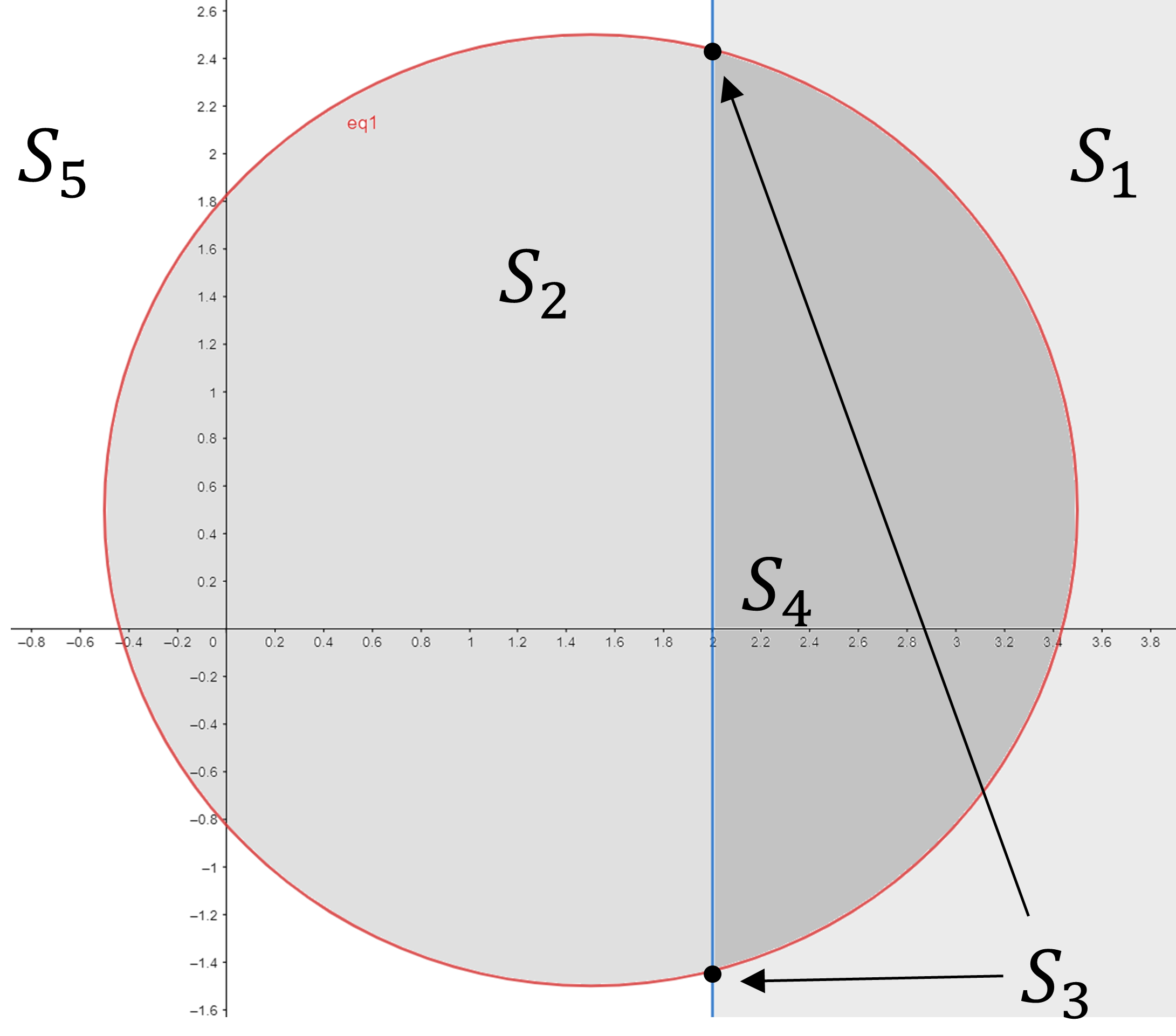}
  \caption{Partition of $\C$}
 \label{fig1}
\end{figure}
It can be easily seen that if $\,\lambda \in S_1$ (respectively, $\,\lambda \in S_2$, $\,\lambda \in S_3$) then one has $u_q(r)<v_q(r)$ (respectively $u_q(r)>v_q(r)$, $u_q(r)=v_q(r)$) for all $r>0$. In the case $\lambda \in S_4\bigcup S_5$, the relation between $u_q(r)$ and $v_q(r)$ depends on~$r$.

\begin{theorem}\label{th-phi-estim}
Fix $r>0$. Then the extremal functions for $\Phi(\cdot ,\lambda)$ over $\JJ_r$ can be described as follows.
\begin{itemize}
 \item [(1)] If $\,\lambda \in S_1$, then the only extremal function is the resolvent of $f_1$, see \eqref{Gr-rReg-gt1}.
 \item [(2)] If  $\,\lambda \in S_2$, then the only extremal function is the resolvent of $f_2.$
  \item [(3)] If $\,\lambda \in S_3$, then the only extremal function is the resolvent of $f_3.$
  \item [(4)] If $\,\lambda \in S_4$ then
   \begin{itemize}
     \item [] for $r<\mu$, the only extremal function is the resolvent of~$f_1;$

    \item  [] for $r>\mu$, the only extremal function is the resolvent of~$f_2;$

 \item  [] for $r=\mu$, the only extremal function is the resolvent of~$f_3.$
  \end{itemize}

   \item [(5)] If $\,\lambda \in S_5$ then
   \begin{itemize}
 \item  [] for $r>\mu$, the only extremal function is the resolvent of $f_1;$
    \item [] for $r<\mu$, the only extremal function is the resolvent of $f_2;$
     \item  [] for $r=\mu$, the only extremal function is the resolvent of $f_3.$
  \end{itemize}

\end{itemize}
\end{theorem}

\begin{proof}
It follows from \eqref{Phi-1} in Proposition~\ref{coro-resG} and notations \eqref{u-v} that
\begin{equation*}\label{proof-th-phi-estim1}
 \left|\Phi(G_r,\lambda) \right| \le \left\{
 \begin{array}{cc}
 v_q(r), & \text{ if } \quad w_q(r) > 1,\vspace{2mm}\\
 u_q(r), & \text{ if } \quad w_q(r) \leq 1.
 \end{array}
 \right.
\end{equation*}
Thus we have to verify whether (and where) $w_q(r)$ is greater than $1$, or conversely. This verification results in explicit relation
\begin{equation}\label{proof-th-phi-estim3}
 \left|\Phi(G_r,\lambda) \right| \le \left\{
 \begin{array}{cc}
v_q(r), & \text{ if } \,\, 2-\Re \lambda < r   \Re {q}\left(|\lambda-c|^2-\rho^2\right),\vspace{2mm}\\
u_q(r), & \text{ if } \,\, 2-\Re \lambda \geq r   \Re {q}\left(|\lambda-c|^2-\rho^2\right).
 \end{array}
 \right.
\end{equation}

Consider now the three cases: (a) $2-\Re \lambda < r   \Re {q}\left(|\lambda-c|^2-\rho^2\right)$,\linebreak (b) $2-\Re \lambda > r   \Re {q}\left(|\lambda-c|^2-\rho^2\right)$, and (c) $2-\Re \lambda = r   \Re {q}\left(|\lambda-c|^2-\rho^2\right).$

Theorem~\ref{tm-F-S-estim} with $\alpha=2r\Re q$ and $\beta =1+rq$ implies that  estimate \eqref{proof-th-phi-estim3} is sharp. Furthermore, the same theorem describes all of the extremal cases. Namely,
 \begin{equation*}\label{proof-th-phi-estim11}
\hspace{-0.5mm} \left|\Phi(G_r,\lambda) \right|\!=\!\! \left\{
 \begin{array}{cc}
v_q(r), & \hspace{-4mm}\text{ in case  (a) only when } G \text{ is the resolvent of } f_1,\vspace{1mm}\\
u_q(r), &\hspace{-4mm} \text{ in case  (b) only when } G \text{ is the resolvent of } f_2,\vspace{1mm}\\
u_q(r), &\hspace{-4mm} \text{ in case  (c) only when } G \text{ is the resolvent of }  f_3,
 \end{array}
 \right.\hspace{-5mm}
 \end{equation*}
(cf. Example~\ref{ex-gen-calc}).  We now notice that
\begin{itemize}
\item the case (a) occurs for all $r>0$ whenever $\lambda\in S_1$, for $r<\mu$ whenever $\lambda\in S_4$,  and for $r>\mu$ whenever $\lambda\in S_5$;
\item the case (b) occurs for all $r>0$ whenever $\lambda\in S_2$, for $r>\mu$ whenever $\lambda\in S_4$,  and for $r<\mu$ whenever $\lambda\in S_5$;
\item the case (c) occurs for all $r>0$ whenever $\lambda\in S_3$ and for $r=\mu$ whenever $\lambda\in S_4\cup S_5$.
\end{itemize}
Combining these facts, we complete the proof.
\end{proof}

\medskip

Now we pass to the set of {\sl all} resolvents
$$\JJ:=\bigcup_{r>0}\JJ_r.$$

Our aim is to solve Problems 1 and 2 over $\JJ$. For simplicity we concentrate on the case $q=1.$

We first establish estimates and extremal functions for coefficients  $b_2$~ and~$b_3$.

\begin{theorem}\label{pr-5-4}
Let $G\in\JJ$. Then
 \begin{itemize}
   \item [(a)] $|b_2| \le \frac{8}{27}$ with equality only if
   $G$ is the resolvent of~$f_1$ and belongs to $\JJ_{\frac{1}{2}}$. In this case
    $G(z)=\displaystyle\frac{2z}{ze^{i\theta}+\frac32+ \sqrt{(ze^{i\theta}-\frac32)^2+4 ze^{i\theta}}}\,$, hence $|b_3| = \frac{16}{243}\,.$
    \vspace{1mm}
   \item [(b)] $|b_3| \le \frac{27}{128}$ with equality only if $G$  is the resolvent of $f_2$ and belongs to $\JJ_{\frac{1}{3}}$. In this case $b_{2n}=0$.
 \end{itemize}
\end{theorem}

\begin{proof}
Let $G \in \JJ_r$ for some $r>0$.  Inequality \eqref{b2est} implies $$\left|b_2 \right| \leq \max_{r>0}\frac{2r}{(1+r)^3}=\frac8{27}\,,$$
where the maximum is attained at $r=\frac12\,.$ Otherwise $\frac{2r}{(1+r)^3}<\frac8{27}.$

Furthermore, it follows from Theorem~\ref{propo-estim-r} that equality $|b_2|=\frac8{27}$ holds only if $G$ is the resolvent of $f_1(z)=z\frac{1+ze^{i \theta}}{1-ze^{i \theta}}$, $\theta \in\R$. Then $b_2=\frac8{27}e^{i\theta}$ and $b_3=\frac{16}{243}e^{2i\theta}$. In addition, $G\in\JJ_{\frac{1}{2}}$ is defined by formula \eqref{Gr-rReg-gt1} with $q=1$ and $r=\frac{1}{2}$. This proves assertion (a).

Further, $\left|b_3 \right|\leq  \frac{2r}{(1+r)^5}\! \cdot \! \max \left( 1+r,|3r-1|\right)$ by \eqref{b3est}.
 For $r> 1$ we have $\left|b_3 \right|\leq  \frac{2r(3r-1)}{(1+r)^5}$. Since $\frac{2r(3r-1)}{(1+r)^5}$ is a decreasing function, $\left|b_3 \right|< \frac{1}{8}$ for all $r> 1$.
  For $0<r\leq 1$ we have $\left|b_3 \right|\leq  \frac{2r}{(1+r)^4}$. Then $\max \frac{2r}{(1+r)^4}$ is attained at $r=\frac{1}{3}\,$ and equals $\frac{27}{128}$. By statement (ii) of Theorem~\ref{propo-estim-r},  the extremal resolvent $G \in \J_{\frac{1}{3}}$ fits to $f_2(z)=z\frac{1+e^{i \theta} z^2}{1-e^{i \theta} z^2}$. This implies $G$ is odd, so $b_{2n}=0$ for all $n.$
\end{proof}

Next we obtain estimate on the Fekete--Szeg\"{o} functional $\Phi(\cdot,\lambda)$ over $\JJ$.
 Note that for $q=1$ the function $u_q(r)=\frac{2r}{(1+r)^5}$ attains the maximal value $k=\frac{1}{2}\cdot \left(\frac{4}{5}\right)^5$ at the point $r=\frac{1}{4}$.

\begin{theorem}\label{th-appl for fekete}
Let $G\in\JJ$. Then
\begin{equation*}\label{fekete-sego1}
\left|\Phi(G,\lambda)\right|\leq k\cdot \max(1, |2\lambda-3|)
\end{equation*}
for all $\lambda \in \C$. Moreover,
  the estimate $\left|\Phi(G,\lambda)\right|\leq k$ is sharp for every $\lambda$ such that $|2\lambda-3|\le1$.
\end{theorem}
\begin{proof}
 Substituting $q=1$ in \eqref{c-ro-mu}, we get $c=\frac{3}{2}$, $\rho=\frac{1}{2}$ and $\mu=\frac{4(2-\Re \lambda)}{\left|2\lambda-3\right|^2-1}\,$.

Now we relate to the partition of $\C$ introduced in \eqref{partition}.
 First, let $\lambda \in S_1=\left\{\lambda \in \C : \Re \lambda \geq2 \right\} \setminus\{ 2\}$. Then case (a) from the proof of Theorem~\ref{th-phi-estim} occurs. Therefore $\left|\Phi(G,\lambda)\right|\leq v_1(r)$ by \eqref{proof-th-phi-estim3}.
 Let us estimate $ v_1(r)$ in \eqref{u-v} as follows
\begin{eqnarray*}\label{v-estim1}
v_1(r) &=&u_1(r)\cdot  \left| 1+\frac{2r}{r+1}(\lambda -2)\right| \\
&=& u_1(r) \left(1+\frac{r}{r+1}\cdot 4 \Re (\lambda - 2)+\left(\frac{r}{r+1}\right)^2\cdot 4|\lambda-2|^2\right)^{\frac{1}{2}} \\
  &<& u_1(r) \left(1+4 \Re (\lambda - 2)+ 4|\lambda-2|^2\right)^{\frac{1}{2}}= u_1(r)\cdot |2\lambda-3|.
 \end{eqnarray*}

Thus, $\left|\Phi(G,\lambda)\right|< k\cdot |2\lambda-3|$ for all $r>0$.

 Second, let $\lambda \in S_5=\{\lambda \in \C : \Re \lambda <2 \text{ and } \lambda \notin \overline{D_{\frac{1}{2}}(\frac{3}{2})}\}$. By the proof of Theorem~\ref{th-phi-estim}, for $0<r \leq\mu$  inequality $\left|\Phi(G,\lambda)\right|\leq u_1(r)$ holds, while for  $r>\mu$  inequality  $\left|\Phi(G,\lambda)\right|\leq v_1(r)$ holds.

We claim that $\left|\Phi(G,\lambda)\right| <k\cdot |2 \lambda -3|$ for all $r>0$.
The notations $\nu:=\frac{1}{2-\lambda}$, $a=\Re \nu $ and $b=\Im \nu$ is convenient to estimate $ v_1(r)$. Indeed, for all $\lambda \in S_5$ we have $\mu=\frac{a}{1-a}$ and $0<a<1$.  Remark also, that $r>\mu$ implies $\frac{r}{r+1}-a>0$. Then
\begin{eqnarray*}\label{v-estim3}
v_1(r) &=& u_1(r) \cdot  \left| 1+\frac{2r}{r+1}(\lambda -2)\right|\\
&=& u_1(r) \left(1+\frac{r}{r+1}\cdot \frac{4}{a^2+b^2}\left(\frac{r}{r+1}-a\right)\right)^{\frac{1}{2}}\\
& < & u_1(r) \left(1+\frac{4}{a^2+b^2}\left(1-a\right)\right)^{\frac{1}{2}}= u_1(r)\cdot |2 \lambda -3|.
 \end{eqnarray*}
Thus, for all  $\lambda \in S_5$ we have $\left|\Phi(G,\lambda)\right|\leq k\cdot |2\lambda-3|$.

Further, if  $\lambda\in S_2=\{\lambda \in \C : |2\lambda-1| \le 3\}\setminus\{2\}$ (respectively, $\lambda=2$), then by Proposition~\ref{coro-resG} and  the proof of Theorem~\ref{th-phi-estim}, $\left|\Phi(G,\lambda)\right|\leq u_1(r)$ for all $r>0$ with equality only when $G$ is the resolvent of $f_2$ (respectively, of $f_3$) and belongs to $\JJ_\frac{1}{4}$.
 Thus for $|2\lambda-3| \leq 1$  the estimate $\left|\Phi(G,\lambda)\right|\leq\max u_1(r)= k$ is sharp.

Since $S_4=\emptyset$ as $q=1$, the proof is complete.
 \end{proof}

{\it Inter alia}, Theorem~\ref{th-appl for fekete} includes the sharp estimate on the
Hankel determinant $H_1^2$ over $\JJ$,
namely, $|H_1^2(G)|\le  k$.

\begin{remar}\label{rem-S1}
We emphasize that our estimate is sharp for $\lambda \in S_2\cup S_3$, while for  $\lambda \in S_1\cup S_5$ it can be improved.

For instance, let $\lambda \in S_1$. Denote $a:=\Re (\lambda-2)\ge0$ and $b:=|\lambda-2|^2\ge a^2$. We already know that $|\Phi(G,\lambda)|\le v_1(r)$ as $G\in\JJ_r$. Hence we have to estimate $\sup\limits_{r>0}v_1(r).$
Consider the function $v(t):=v_1\left(\frac{t}{1-t}\right)=2t(1-t)^4|1+2t(\lambda-2)|$ with $t=\frac{r}{1+r}\in(0,1)$.
A straightforward calculation leads to
$$
\left(\frac{v(t)^2}2\right)'\!\!=\! 2t(1-t)^7\left( 1-5t -24bt\phi(t) \right))\ \mbox{ with }\ \phi(t)\!=\!t^2-\left(\frac13-\frac{11a}{12b} \right)t -\frac{a}{4b}\,.
$$
Denote $t_1:=\frac12\left[ \frac13- \frac{11a}{12b} +\sqrt{\left(\frac13- \frac{11a}{12b}\right)^2 +\frac ab} \right],$ the largest root of $\phi$. One sees that $\frac14<t_1<1$ and concludes that $v$ is a decreasing function as $t>t_1.$ Therefore the problem concerns $\sup\limits_{0<t<t_1}v(t).$ Since $t(1-t)^4$ takes its maximal values $k$ at $t=\frac15$, while $|1+2t(\lambda-2)|$ is an increasing function, we summarize that $v(t)\le k\cdot |1+2t_1(\lambda-2)|<k\cdot|2\lambda-3|$.
\end{remar}

\medskip

To complete the paper, recall that every element of any resolvent family (that is, each $G\in\JJ$) is a univalent self-mapping of the open unit disk that preserves zero. Usually, in geometric function theory different coefficient functionals are studied over families of  {\it normalized} univalent functions. By this reason we introduce now the class of {\it normalized resolvents}
\[
\widetilde{\JJ}:=\left\{g\in\Hol(\D,\C):\   \frac g{1+r}\in \JJ_r\quad \text{for some }r>0\right\}
\]
(obviously, if $g\in\widetilde{\JJ}$, then $g(0)=g'(0)-1=0$) and solve Problems 1 and~2 over $\widetilde{\JJ}$.

\begin{theorem}\label{th-normalized1}
  Let $g\in \widetilde{\JJ}$ with $g(z)=z+b_2z^2+b_3z^3+\ldots$. Then
 \begin{itemize}
   \item [(a)] $|b_2| \le \frac{1}{2}$ with equality only if
   $\frac{g}{2}$ is the resolvent of~$f_1$, in which case $|b_3| = \frac{1}{4}$.
   \item [(b)] $|b_3| \le \frac{8}{27}$  with equality only if
   $\frac{2g}{3}$ is the resolvent of~$f_2$, in which case $|b_{2n}| = 0$.
  \item [(c)] $\left|\Phi(g,\lambda) \right| \le \frac{8}{27}\cdot \max(1,|2\lambda-3|)$ for all $\lambda \in \C$.
 \end{itemize}
\end{theorem}
Consequently, the inequality $\left|H^2_1(g) \right| \le \frac{8}{27}$ is sharp.

The proof of Theorem~\ref{th-normalized1} is similar to those of Theorems~\ref{pr-5-4} and~ \ref{th-appl for fekete}.

\end{document}